\documentclass{amsart}[13pt]

\usepackage{amsmath, amsthm, amssymb, amsfonts}
\usepackage{blkarray}
\usepackage{graphicx}

\usepackage{ytableau}
\usepackage{mathtools}
\ytableausetup{aligntableaux=top}
\ytableausetup{smalltableaux}
\ytableausetup{mathmode, boxsize=10pt}

\newtheorem{thm}{Theorem}[section]
\newtheorem{cor}[thm]{Corollary}
\newtheorem{prop}[thm]{Proposition}
\newtheorem{lem}[thm]{Lemma}

\theoremstyle{remark}
\newtheorem{rem}[thm]{Remark}
\theoremstyle{definition}
\newtheorem{defn}[thm]{Definition}

\newcommand{\FI}{\mathrm{FI}}
\newcommand{\OI}{\mathrm{OI}}

\newcommand{\Vect}{\mathrm{Vec}}

\newcommand{\Hom}{\mathrm{Hom}}

\newcommand{\Tableaux}{\mathrm{Tableaux}}
\newcommand{\F}{F}

\DeclareMathAlphabet{\functor}{OT1}{pzc}{m}{it}

\newcommand{\fn}[1]{\vphantom{\scalebox{1.3}{\mbox{\fbox{\tt{1}}}}}\raisebox{1.5pt}{\scalebox{.7}{\mbox{\fbox{\tt{#1}}}}}}

\title[On computing the eventual behavior of an $\FI$-module over $\mathbb{Q}$]{On computing the eventual behavior of \\ an $\FI$-module over the rational numbers}
\author{John D. Wiltshire-Gordon}
\begin{document}
\begin{abstract}
We give a formula for the eventual multiplicities of irreducible representations appearing in a finitely-presented $\FI$-module over the rational numbers.  The result relies on structure theory due to Sam-Snowden \cite{SS16}.
\end{abstract}

\maketitle
\section{Introduction}
\noindent
Let $\FI$ be the category whose objects are the finite sets $[n] = \{1, \ldots, n\}$ for $n \in \mathbb{N}$ and whose morphisms are injections.  An \textbf{$\FI$-module over the rational numbers} is a functor $\FI \to \Vect_{\mathbb{Q}}$.

A familiar example is the free $\mathbb{Q}$-vector space functor $\F^1 \colon \FI \to \Vect_{\mathbb{Q}}$ given by
$$
\F^1[n] = \mathbb{Q}^n;
$$
an even simpler example is the constant functor $\F^0[n] = \mathbb{Q}$.  As suggested by the superscript, these two modules are part of a family.  In general, $\F^k[n] = \mathbb{Q} \FI(k, n)$, where $\FI(k, n)$ denotes the set of injections $[k] \to [n]$.  The $\F^k$ are called ``free $\FI$-modules'' for reasons that we explain in \S\ref{sec:structure}.

The free $\FI$-modules are combinatorially straightforward.  To build a running example of more realistic difficulty, define the vector space
$$
E[n] = \mathbb{Q} \cdot \{ \mbox{symbols $z_{ijk}$ with $i, j, k \in [n]$ distinct} \} / (z_{ijk} + z_{jkl} + z_{kli} + z_{lij} = 0),
$$
noting that an injection $f \colon [x] \to [y]$ induces a linear map $E[x] \to E[y]$ by the rule
$$
z_{ijk} \mapsto z_{f(i) \, f(j) \, f(k)}
$$
so that $E$ is an $\FI$-module.  For $n \leq 10$, a direct computation gives 
$$
\begin{array}{rccccccccccc}
n=& 0 & 1 & 2 & 3 & 4 & 5 & 6 & 7 & 8 & 9 & 10 \\
E[n] = & 0 & 0 & 0 & \mathbb{Q}^6 & \mathbb{Q}^{18} & \mathbb{Q}^{30} & \mathbb{Q}^{44} & \mathbb{Q}^{56} & \mathbb{Q}^{76} & \mathbb{Q}^{99} & \mathbb{Q}^{125}, \\
\end{array}
$$
a sequence with growing dimension.  However, accounting for the symmetric group $\mathfrak{S}_n$ action on the space $E[n]$ gives a sequence of representations:
\ytableausetup{mathmode, boxsize=2pt}
\ytableausetup{aligntableaux=center}
$$
\begin{array}{rccccccccccc}
n=& 0 & 1 & 2 & 3 & 4 & 5 & 6 & 7 & 8 & 9 & 10 \\
 &  &  &  & \ydiagram{3} &  &  &  &  & & & \\
 &  &  &  & 2\ydiagram{2,1} & 3\ydiagram{3,1} & 2\ydiagram{4,1} & 2\ydiagram{5,1} & 2\ydiagram{6,1} & 2\ydiagram{7,1} & 2\ydiagram{8,1} & 2\ydiagram{9,1} \\
  &  &  &  & \ydiagram{1,1,1} & \ydiagram{2,2} & 2\ydiagram{3,2} & \ydiagram{4,2} & \ydiagram{5,2} & \ydiagram{6,2} & \ydiagram{7,2} & \ydiagram{8,2} \\
   &  &  &  &  & 2\ydiagram{2,1,1} & 2\ydiagram{3,1,1} & 2\ydiagram{4,1,1} & 2\ydiagram{5,1,1} & 2\ydiagram{6,1,1} & 2\ydiagram{7,1,1} & 2\ydiagram{8,1,1} \\
 &  &  &  &  & \vphantom{\ydiagram{1,1,1,1,1}} \ydiagram{1,1,1,1} &  &  &  &  &  &  \\
  &  &  &  &  &  &  & \ydiagram{3,3} &  & &  &  \\
\end{array}
$$
where we have used the usual indexing of irreducible representations by partitions.  A stabilization pattern is now apparent: just add more boxes in the top row.

Church-Farb named this phenomenon \textbf{representation stability}, observing it in several contexts \cite{CF13}.  Later, in work with Ellenberg, these authors introduced $\FI$-modules as a firmer algebraic foundation \cite{CEF15}.

The goal of this paper is to provide a formula for the limiting multiplicities as $n \to \infty$, combining structure theory due to Sam-Snowden \cite{SS16} with ideas appearing in this author's dissertation \cite{WG16}.  In \S\ref{sec:compute}, we will be able to prove computationally the limiting multiplicities of the $\FI$-module $E$:
$$
\begin{array}{cccccccc}
\mu(\ydiagram{1}^+, E) = 2 & &
\mu(\ydiagram{2}^+, E) = 1 & &
\mu(\ydiagram{1,1}^+, E) = 2 & &
\mu(\lambda^+, E) = 0 & \mbox{otherwise}, \\
\end{array}
$$
where the superscript $+$ stands for an invisible long top row.  We make these multiplicities precise in Definition \ref{defn:mu}.

Theorem \ref{thm:main} gives a formula for $\mu(\lambda^+, M)$ for any finitely-presented $\FI$-module $M$ in terms of the corank of a certain combinatorial matrix construction $\mathbb{A}_{\lambda}$ applied to any presentation matrix for $M$.  The result is similar to \cite[Theorem 4.3.5]{WG16}, which applies in the context of categories of dimension zero; see \cite{WG15}.  However, Theorem \ref{thm:main} does not follow directly because $\FI$ is dimension one.  
\subsection{Finitely presented $\FI$-modules}
\noindent
To compute the eventual multiplicities of an $\FI$-module $M$, Theorem \ref{thm:main} requires as input a presentation matrix, describing $M$ as the cokernel of a map between direct sums of free modules.  
The Noetherian property of $\FI$-modules guarantees that any finitely-generated $\FI$-module may be written in this fashion with a finite matrix.  The precise setup will be given in \S\ref{sec:structure}.

For now we say that a presentation matrix takes the form
$$
\begin{blockarray}{ccccc}
 & y_1 & y_2 & \cdots & y_r \\
\begin{block}{c[cccc]}
 x_1  &  &  & & \\
 x_2  &  &  & & \\
\vdots & & & & \\
 x_g  & & & & \\
\end{block}
\end{blockarray}
$$
for some generation degrees $x_1, \ldots, x_g \in \mathbb{N}$ and relation degrees $y_1, \ldots y_r \in \mathbb{N}$, and that the entry in position $(i, j)$ is a formal linear combinations of injections $[x_i] \to [y_j]$.  Each column of a presentation matrix imposes some relation.  The $\FI$-module $E$ is presented by the $1 \times 1$ matrix
$$
\begin{blockarray}{cc}
 & 4 \\
\begin{block}{c[c]}
 3  & \fn{123} + \fn{234} + \fn{341} + \fn{412}  \\
\end{block}
\end{blockarray},
$$
where, for example, we have written $\fn{412} \in \FI(3,4)$ for the injection
$$
\begin{array}{rcl}
1 & \mapsto & 4 \\
2 & \mapsto & 1 \\
3 & \mapsto & 2, \\
\end{array}
$$
and the other injections are similar.

\subsection{Tableau combinatorics}
 Write $\mathbb{N}_+ = \{1, 2, 3, \ldots \}$ for the poset of positive natural numbers, and $
\mathbb{N}_+^2$ for the product poset, where the ordering is given by $(i, j) \leq (i',j')$ if $i \leq i'$ and $j \leq j'$.  Our pictures of subsets of $\mathbb{N}_+^2$ use matrix coordinates so that order ideals are upper-left justified.

An order ideal of the product poset $\mathbb{N}_+^2$ is called a \textbf{diagram}.  
An injection $t \colon [k] \to \mathbb{N}_+^2$ is called a \textbf{tableau} if, for all $m \leq k$, the subset $t([m]) \subset \mathbb{N}_+^2$ is a diagram.  In particular, $\mathrm{im}(t)$ is a diagram---the \textbf{shape} of $t$.

If $\lambda$ is a diagram of size $k$, write $|\lambda|=k$.  The collection of tableaux with shape $\lambda$ will be written
$$
\Tableaux(\lambda) = \{ t \colon [k] \to \mathbb{N}_+^2 \mbox{ so that $\mathrm{im}(t) = \lambda$} \}.
$$
The row lengths of $\lambda$, written $\lambda_1, \lambda_2, \ldots$, form a non-increasing sequence with $k = \lambda_1 + \lambda_2 + \cdots $.  In this way, $\lambda$ may be considered an integer partition of $k$.
\noindent
\subsection{Three combinatorial functions: $\zeta$, $\xi$, and $\chi$.}
We describe three functions needed in the construction of the matrix $\mathbb{A}_{\lambda}(f)$, and that therefore appear indirectly in the statement of Theorem \ref{thm:main}.  

The lexicographic ordering on the elements of $\lambda$ provides a distinguished bijection $t_{\lambda} \colon [k] \to \lambda$.
Each element $t \in \Tableaux(\lambda)$ relates to this distinguished element by means of a unique permutation $\zeta(t) \in \mathfrak{S}_k$ satisfying $t \circ \zeta(t) = t_{\lambda}$.  Similarly, if $p \in \FI(k,n)$ is an injection, write $\xi(p) \in \mathfrak{S}_k$ for the unique permutation with the property that $p \circ \xi(p)^{-1}$ is monotone.

The definition of $\chi$ is slightly more involved.  Given functions $a, b \colon [k] \to \mathbb{N}_+$, write $(a,b) \colon [k] \to \mathbb{N}_+^2$ for $l \mapsto (a(l), b(l))$.  Let
$$
\chi(a, b) = \begin{cases*}
      (-1)^{\zeta(t)} & if $(a, b) \colon [k] \to \mathbb{N}_+^2$ defines a valid tableau $t$ \\
      0        & otherwise,
    \end{cases*}
$$
where $(-1)^{\sigma}$ denotes the sign of a permutation $\sigma \in \mathfrak{S}_k$.

\subsection{Construction of $\mathbb{A}_{\lambda}(f)$ for $f \colon [x] \to [y]$ an injection and $\lambda$ a diagram}
Let $k = |\lambda|$, write $\OI(k,n)$ for the set of monotone injections $[k] \to [n]$, and write $r, c \colon \mathbb{N}_+^2 \to \mathbb{N}_+$ for the two projection maps $r(i,j) = i$ and $c(i,j) = j$.  We construct a matrix $\mathbb{A}_{\lambda}(f)$ with 
$$
\begin{array}{crr}
& \mbox{rows indexed by pairs} & (p, t) \in \OI(k,x) \times \Tableaux(\lambda) \phantom{.} \\
\mbox{and} & \mbox{columns indexed by pairs} & (q, u) \in \OI(k,y) \times \Tableaux(\lambda). \\
\end{array}
$$
The entry in position $((p,t), (q,u))$ is the rational number given by the formula
$$
\mathbb{A}_{\lambda}(f)_{(p,t), (q,u)} = \begin{cases*} 
\chi(r \circ u, \, c \circ t \circ \xi(f \circ p)) & if $f \circ p$ and $q$ have the same image \\
0 & otherwise.
\end{cases*}
$$
Extend the definition of $\mathbb{A}_{\lambda}$ linearly to formal combinations of injections:
$$
\begin{array}{ccc}
\mathbb{A}_{\lambda}(f + g) = \mathbb{A}_{\lambda}(f) + \mathbb{A}_{\lambda}(g) & & \mathbb{A}_{\lambda}(\alpha f) = \alpha \mathbb{A}_{\lambda}(f).
\end{array}
$$
\subsection{The main result}
Suppose that $Z$ is a presentation matrix for an $\FI$-module $M$ with generation degrees $x_1, \ldots, x_g \leq x_{\max}$ and relation degrees $y_1, \ldots y_r \leq y_{\max}$.  Recall that the corank of a rational matrix is its row-count minus its rank.

\begin{thm} \label{thm:main}
For any diagram $\lambda$, the multiplicity of $\lambda^+$ in $M$ is given by
$$
\mu(\lambda^+, M) = \mathrm{corank} \left( \, \mathbb{A}_{\lambda}Z \right)
$$
where $\lambda^+$ denotes the diagram obtained by attaching a long top row to $\lambda$, and $\mathbb{A}_{\lambda}Z$ denotes the rational block matrix obtained by applying the construction $\mathbb{A}_{\lambda}$ to the entries of the presentation matrix $Z$.  In particular, $\mu(\lambda^+, M) = 0$ whenever $|\lambda| > x_{\max}$, since in this case $\mathbb{A}_{\lambda}Z$ has no rows.
\end{thm}
\begin{cor}[Eventual invariants] \label{cor:invariants}
As $n \to \infty$, the eventual multiplicity of the trivial $\mathfrak{S}_n$-representation is
$$
\mu(\emptyset^+, M) = \mathrm{corank} \, \left( \varepsilon Z \right),
$$
where $\varepsilon Z$ is the matrix obtained from $Z$ by replacing every injection with $1 \in \mathbb{Q}$.
\end{cor}
\begin{proof}
The construction $\mathbb{A}_{\emptyset}$ coincides with $\varepsilon$.
\end{proof}
\noindent
In the next statement, set $l_i = \lambda_i + |\lambda| - i$ for $i \in \{1, \ldots, |\lambda|\}$.

\begin{cor}[Eventual dimension] \label{cor:dims}
The sequence $n \mapsto \dim_{\mathbb{Q}} M$ eventually agrees with the polynomial
$$
n \mapsto \sum_{\lambda} \mathrm{corank} \left( \, \mathbb{A}_{\lambda}Z \right) \cdot \frac{\prod_{i < j} (l_i - l_j)}{\prod_i (l_i)!} \cdot \prod_i (n-l_i)
$$
where the sum ranges over all $\lambda$ with $|\lambda| \leq x_{\max}$, and $i, j \in \{1, \ldots, |\lambda|\}$.
\end{cor}
\begin{proof}
By the Frobenius character formula; see \cite[(4.11)]{FH04}.
\end{proof}

\begin{rem} \label{rem:onset}
The onset of stabilization is mostly understood; see \cite[Theorem~3.3.4]{CEF15}, \cite[Theorem A]{CE17}, \cite[Remark 7.4.6]{SS16}, \cite[Theorem A]{Ramos15}, for example.  In particular, by \cite[Theorem A]{CE17}, the regularity of $M$ is bounded above by $x_{\max} + y_{\max} - 1$, and by \cite[Theorem 1.1]{NSS17}, the regularity is bounded below
$$
\max_i (i+\deg H^i_{\mathfrak{m}} M) \leq \mathrm{reg}(M),
$$
so we must have $\deg H^i_{\mathfrak{m}} M = -\infty$ once $i \geq x_{\max} + y_{\max}$, and this means that the polynomial named $q$ in \cite[Theorem 5.1.3]{SS16} has degree at most $x_{\max} + y_{\max} - 1$ by \cite[Proposition 5.3.1]{SS16}.  Therefore, the eventual multiplicities computed in Theorem \ref{thm:main} are attained once $n \geq x_{\max}+ y_{\max} $.
\end{rem}

\subsection{Structure of this paper}
In \S\ref{sec:compute}, we use Theorem \ref{thm:main} to compute the eventual multiplicities for the example $\FI$-module $E$.  In \S\ref{sec:structure}, we discuss finitely presented $\FI$-modules and recall some of their structure theory, especially results due to Sam-Snowden and Nagpal.  We highlight the role of ``induced'' $\FI$-modules, by which we mean those left-Kan-extended from the symmetric groups.  In \S\ref{sec:induced}, we provide an explicit description of these induced modules by their action matrices.  In \S\ref{sec:infinite}, we introduce the infinite diagram $\lambda^+$ and make precise some ideas that had been treated intuitively.  In \S\ref{sec:proof} we give the proof of Theorem \ref{thm:main}.  Finally, in \S\ref{sec:code}, we provide computer code making Theorem \ref{thm:main} algorithmic.  This code relies on the tableaux-generation routines provided by the computer algebra software Sage \cite{sagemath}.

\subsection{Acknowledgements}
A preliminary form of these results were presented at the American Institute of Mathematics in July 2016.  Thanks to Eric Ramos for a helpful conversation.  The author acknowledges support from the NSF through grant DMS-1502553.

\section{Example calculation} \label{sec:compute}
We use Theorem \ref{thm:main} to compute the eventual multiplicities in the running example $E$.  Supporting Sage code may be found in \S\ref{sec:code}.  In order to apply Theorem~\ref{thm:main}, we must compute the matrix $\mathbb{A}_{\lambda}Z$ for every $\lambda$, where
$$
Z = \begin{blockarray}{cc}
 & 4 \\
\begin{block}{c[c]}
 3  & \fn{123} + \fn{234} + \fn{341} + \fn{412}  \\
\end{block}
\end{blockarray}.
$$
To this end, we must compute the sum of the four matrices $\mathbb{A}_{\lambda}(\fn{123})$, $\mathbb{A}_{\lambda}(\fn{234})$, $\mathbb{A}_{\lambda}(\fn{341})$, and $\mathbb{A}_{\lambda}(\fn{412})$.  Fortunately, if $|\lambda| > 3$ then the resulting matrices have no rows, and so their coranks are zero.  This leaves a finite number of choices for $\lambda$.

\ytableausetup{aligntableaux=center}
\ytableausetup{mathmode, boxsize=3pt}

We will build two of these matrices by hand, and then use Sage to do the others.  Let $\lambda = \ydiagram{2}$ and $f = \fn{123}$.  
\ytableausetup{mathmode, boxsize=10pt}
There are three monotone injections $[2] \to [3]$, six monotone injections $[2] \to [4]$, and only one tableau of shape $\lambda$, which we name $\theta = \begin{ytableau} 1 & 2 \end{ytableau}$---shorthand for the function $1 \mapsto (1,1), 2 \mapsto (1,2)$.  The matrix $\mathbb{A}_{\lambda}(f)$ is then given by
$$
\begin{blockarray}{ccccccc}
 & \fn{12} \times \theta & \fn{13} \times \theta & \fn{14} \times \theta & \fn{23} \times \theta & \fn{24} \times \theta & \fn{34} \times \theta  \\
\begin{block}{c[cccccc]}
 \fn{12} \times \theta & 1 & 0 & 0 & 0 & 0 & 0 \\
 \fn{13} \times \theta & 0 & 1 & 0 & 0 & 0 & 0 \\
 \fn{23} \times \theta & 0 & 0 & 0 & 1 & 0 & 0 \\
\end{block}
\end{blockarray}.
$$
The other three matrices $\mathbb{A}_{\lambda}(\fn{234})$, $\mathbb{A}_{\lambda}(\fn{341})$, and $\mathbb{A}_{\lambda}(\fn{412})$, have the same format:
$$
\begin{array}{ccc}
\left[
\begin{array}{cccccc}
 0 & 0 & 0 & 1 & 0 & 0 \\
 0 & 0 & 0 & 0 & 1 & 0 \\
 0 & 0 & 0 & 0 & 0 & 1 \\
\end{array}
\right] & 
\left[
\begin{array}{cccccc}
 0 & 0 & 0 & 0 & 0 & 1 \\
 0 & 1 & 0 & 0 & 0 & 0 \\
 0 & 0 & 1 & 0 & 0 & 0 \\
\end{array}
\right] &
\left[
\begin{array}{cccccc}
 0 & 0 & 1 & 0 & 0 & 0 \\
 0 & 0 & 0 & 0 & 1 & 0 \\
 1 & 0 & 0 & 0 & 0 & 0 \\
\end{array}
\right] \\
\end{array}.
$$
\noindent
Consequently,
\ytableausetup{mathmode, boxsize=3pt}
$$
\mathbb{A}_{\ydiagram{2}}\left(\fn{123}+\fn{234}+\fn{341}+\fn{412}\right) = \left[
\begin{array}{cccccc}
 1 & 0 & 1 & 1 & 0 & 1 \\
 0 & 2 & 0 & 0 & 2 & 0 \\
 1 & 0 & 1 & 1 & 0 & 1 \\
\end{array}
\right],
$$
and so $\mu(\ydiagram{2}^+,E) = 1$ by Theorem \ref{thm:main}, since this matrix has corank $1$.

For our second example matrix, let $\lambda = \ydiagram{2,1}$, and set
\ytableausetup{mathmode, boxsize=10pt}
$$
\begin{array}{ccc}
\gamma = \begin{ytableau} 1 & 2 \\ 3 \end{ytableau} & & \delta = \begin{ytableau} 1 & 3 \\ 2 \end{ytableau} \\
\end{array}.
$$
We compute $\mathbb{A}_{\lambda}\left( \fn{123} \right)$:
$$
\begin{blockarray}{ccccccccc}
& \fn{123} \hspace{-2pt} \times \hspace{-2pt} \delta & \fn{123} \hspace{-2pt} \times \hspace{-2pt} \gamma & \fn{124} \hspace{-2pt} \times \hspace{-2pt} \delta & \fn{124} \hspace{-2pt} \times \hspace{-2pt} \gamma & \fn{134} \hspace{-2pt} \times \hspace{-2pt} \delta & \fn{134} \hspace{-2pt} \times \hspace{-2pt} \gamma & \fn{234} \hspace{-2pt} \times \hspace{-2pt} \delta & \fn{234} \hspace{-2pt} \times \hspace{-2pt} \gamma  \\
\begin{block}{c[cccccccc]}
 \fn{123} \hspace{-2pt} \times \hspace{-2pt} \delta & -1 & 0 & 0 & 0 & 0 & 0 & 0 & 0 \\
 \fn{123} \hspace{-2pt} \times \hspace{-2pt} \gamma & 0 & 1 & 0 & 0 & 0 & 0 & 0 & 0 \\
\end{block}
\end{blockarray}.
$$
In order to perform these computations automatically, run the code from \S\ref{sec:code} in a fresh Sage session, and then
\begin{verbatim}
for k in range(4):
    for shape in Partitions(k):
        injections = [[1, 2, 3], [2, 3, 4], [3, 4, 1], [4, 1, 2]]
        coefficients = [1, 1, 1, 1]
        AAZ = sum([alpha * AA(shape, 3, f, 4)
                   for alpha, f in zip(coefficients, injections)])
        print "corank AA_" + str(shape) + "(Z) = " + \
              str(AAZ.nrows() - AAZ.rank())
\end{verbatim}
producing the output
\begin{verbatim}
corank AA_[](Z) = 0
corank AA_[1](Z) = 2
corank AA_[2](Z) = 1
corank AA_[1, 1](Z) = 2
corank AA_[3](Z) = 0
corank AA_[2, 1](Z) = 0
corank AA_[1, 1, 1](Z) = 0
\end{verbatim}
matching the table of low degrees given in the introduction.
From Corollary \ref{cor:dims}, we conclude that
\begin{align*}
\mathrm{dim}_{\mathbb{Q}} E(n) & = 2\left( n-1 \right) + n(n-3)/2+2(n-1)(n-2)/2 \\
&= n(3n-5)/2.
\end{align*}
for all $n \gg 0$.  In fact, using Remark \ref{rem:onset}, it suffices to take $n \geq 7$.

\section{Background on the structure theory of $\FI$-modules} \label{sec:structure}
We explain free $\FI$-modules and describe the sort of matrix that defines a map between frees.  We then discuss some of the basic theory of $\FI$-modules.
\subsection{Free $\FI$-modules and Yoneda's lemma}
The free $\FI$-module $\F^k$ is the linearization of the functor represented by $[k] \in \FI$.  Explicitly,
$$
\F^k[n] = \mathbb{Q} \cdot \{\mbox{injections } [k] \to [n] \},
$$
and $\FI$-morphisms act by post-composition.  The free module $F^k$ has a special vector sitting in degree $k$, which is written $1_k$, and stands for the identity injection $[k] \to [k]$.  This vector, $1_k \in \F^k[k]$, is the \textbf{standard basis vector} for the free module $\F^k$ in the same way that the multiplicative identity in a ring is the standard basis vector for the ring as a rank-one free module.  Yoneda's lemma says that, for any $\FI$-module $M$, the map
$$
\Hom(\F^k, M) \longrightarrow M[k]
$$
sending an $\FI$-module map $\varphi \colon \F^k \to M$ to its evaluation $\varphi(1_k) \in M[k]$ is an isomorphism.  In other words, the basis vector $1_k$ may be sent anywhere, and once its destination is determined, the rest of the map $\varphi$ is determined as well.

\subsection{Finitely presented $\FI$-modules}
Suppose $M$ is an $\FI$-module that is generated by vectors $m_1, \ldots, m_g$ where $m_i \in M[x_i]$ for various $x_1, \ldots, x_g \in \mathbb{N}$, possibly with repetition.  By Yoneda's lemma, each element $m_i$ determines a map $\F^k \to M$.  Summing these maps, we obtain
$$
\bigoplus_{i=1}^g \F^{x_i} \longrightarrow M,
$$
which is a surjection since every generator is in its image.  (Specifically, $m_i$ is hit by the standard basis vector $1_{x_i}$.) The kernel of this surjection is an $\FI$-submodule.  

In order for $M$ to be finitely presented, we want this submodule to be finitely generated.  Amazingly, this is always the case by a fundamental property of $\FI$-modules called Noetherianity.  
In the case of $\mathbb{Q}$-coefficients, Noetherianity follows from work of Snowden \cite[Theorem 2.3]{snowden13}; the result for coefficients in $\mathbb{Z}$, or in a general Noetherian ring, is due to Church-Ellenberg-Farb-Nagpal \cite{CEFN14}.

Using Noetherianity, pick a sequence of generators for the kernel $c_1, \ldots, c_r$ with 
$$
c_j \in \left(\bigoplus_{i=1}^g \F^{x_i}\right)[y_j]
$$
for various $y_1, \ldots, y_r \in \mathbb{N}$, once again with repetition permitted.  Projecting each $c_j$ onto each summand $\F^{x_i}$, obtain a collection of ``matrix entries''
$$
z_{ij} \in \F^{x_i}[y_j] = \mathbb{Q}\FI(x_i,y_j)
$$
so that $c_j = \sum_i z_{ij}$.  The $z_{ij}$ then define a matrix $Z$ of the form indicated in the introduction.  Moreover, the module $M$ is the cokernel of the map defined on basis vectors by the rule $1_{y_j} \mapsto c_j$.  This is the sense in which the columns of $Z$ impose relations on generators indexed by the rows.

\subsection{Notation for $\mathfrak{S}_k$-representations}
We introduce the Specht modules and provide a formula for their $\mathfrak{S}_k$-actions by matrices.
Recall the function $\chi$ from the introduction:
$$
\chi(a, b) = \begin{cases*}
      (-1)^{\zeta(t)} & if $(a, b) \colon [k] \to \mathbb{N}_+^2$ defines a valid tableau $t$ \\
      0        & otherwise,
    \end{cases*}
$$
where $(-1)^{\zeta(t)}$ denotes the sign of $\zeta(t) \in \mathfrak{S}_k$, the permutation satisfying $t \circ \zeta(t) = t_{\lambda}$.

If $\lambda$ is a diagram, $|\lambda|= k$, and $\sigma \in \mathfrak{S}_k$, define a $\Tableaux(\lambda) \times \Tableaux(\lambda)$ matrix $\mathbb{W}_{\lambda}(\sigma)$ with $(t, u)$-entry given by the formula
$$
\mathbb{W}_{\lambda}(\sigma)_{t, u} = \chi(r \circ u \circ \sigma, c \circ t).
$$
\ytableausetup{mathmode, boxsize=3pt}
For example, ordering the tableaux of shape $ \lambda = \ydiagram{2,2,1}$ as follows,
\ytableausetup{mathmode, boxsize=10pt}
$$
\Tableaux(\lambda) = \left\{ 
\begin{array}{ccccc}
\begin{ytableau}
1 & 2 \\
3 & 4 \\
5
\end{ytableau} &
\begin{ytableau}
1 & 2 \\
3 & 5 \\
4
\end{ytableau} &
\begin{ytableau}
1 & 3 \\
2 & 4 \\
5
\end{ytableau} &
\begin{ytableau}
1 & 3 \\
2 & 5 \\
4
\end{ytableau} &
\begin{ytableau}
1 & 4 \\
2 & 5 \\
3
\end{ytableau}
\end{array}
\right\},
$$
we have
$$
\mathbb{W}_{\lambda}(1) = \left[
\begin{array}{ccccc}
 1 & 0 & 0 & 0 & 1 \\
 0 & -1 & 0 & 0 & 0 \\
 0 & 0 & -1 & 0 & 0 \\
 0 & 0 & 0 & 1 & 0 \\
 0 & 0 & 0 & 0 & -1 \\
\end{array}
\right].
$$
It is no coincidence that this matrix is invertible over $\mathbb{Z}$.  In fact, the matrices $\mathbb{W}_{\lambda}(\sigma)$ yield a module for a certain connected, two-object groupoid that is equivalent to $\mathfrak{S}_k$.  In order to obtain a module for the symmetric group proper, we perform a standard correction, defining
\begin{equation} \label{eq:specht}
\mathcal{W}(\lambda)(\sigma) = \mathbb{W}_{\lambda}(1)^{-1} \cdot \mathbb{W}_{\lambda}(\sigma).
\end{equation}
\begin{thm}[Alfred Young 1928 \cite{Young77}] \label{thm:specht}
If $\lambda$ is a diagram of size $k$, the assignment $\sigma \mapsto \mathcal{W}(\lambda)(\sigma)$ has the property that, for all $\sigma, \tau \in \mathfrak{S}_k$,
$$
\mathcal{W}(\lambda)(\sigma) \cdot \mathcal{W}(\lambda)(\tau) = \mathcal{W}(\lambda)(\tau \circ \sigma),
$$
and so defines a module for the symmetric group $\mathfrak{S}_k$.  This module is irreducible after tensoring with $\mathbb{Q}$.  Moreover, every isomorphism class of irreducible representation appears exactly once among the $\mathcal{W}(\lambda)$.
\end{thm}
For a modern account of this construction, see \cite{Lascoux01}.  For another perspective on the function $\chi$, see \cite{WWZ17}.

\subsection{Induced $\FI$-modules}
For each $k$, write
$$
\functor{i}_k \colon \mathfrak{S}_k \to \FI
$$
for the inclusion of the symmetric group.  The corresponding restriction operation, written $(\functor{i}_k)^*$, takes an $\FI$-module $M$ to the vector space $M[k]$ together with its natural action of $\mathfrak{S}_k$.  For formal reasons, there exists a left adjoint to restriction, written $(\functor{i}_k)_!$.  This functor is left Kan extension along $\functor{i}_k$.  Its defining universal property says that, for any $\mathfrak{S}_k$-representation $W$,
$$
\Hom((\functor{i}_k)_! W, M) \cong \Hom_{\mathfrak{S}_k}(W, (\functor{i}_k)^* M).
$$
In section \S\ref{sec:induced}, we will give a concrete description of the $\FI$-module $(\functor{i}_k)_! W$ in terms of the $\mathfrak{S}_k$-action on $W$.  

Any module of the form $(\functor{i}_k)_! W$ is called an \textbf{induced module}, and similarly for a direct sum of such modules.  If $W$ is projective, then $(\functor{i}_k)_! W$ is also projective.  
Working over $\mathbb{Q}$, every $\mathfrak{S}_k$-representation is projective, and so we obtain a nice class of projective modules
$$
\mathcal{M}(\lambda) = (\functor{i}_k)_! \mathcal{W}(\lambda)
$$
for any digram $\lambda$ with $|\lambda| = k$.

\begin{prop} \label{prop:summand}
The free module $\F^k$ is an induced $\FI$-module.  Specifically,
$$
\F^k \cong (\functor{i}_k)_! \left( \mathbb{Q} \mathfrak{S}_k \right).
$$
Consequently, and making use of $\mathbb{Q}$ coefficients, the free module $\F^k$ decomposes as a direct sum
$$
\F^k \cong \bigoplus_{|\lambda| = k} \mathcal{M}(\lambda)^{\oplus \Tableaux(\lambda)}.
$$
\end{prop}
\begin{proof}
Write $\functor{j}_k \colon \ast \to \FI$ for the functor from the terminal category that picks out the object $[k]$.  Yoneda's lemma gives that $\F^k$ satisfies the universal property defining $(\functor{j}_k)_! \mathbb{Q}$.  On the other hand, the functor $\functor{j}_k$ factors as the composite $\functor{i}_k \circ \functor{h}_k$ where $\functor{h}_k \colon \ast \to \mathfrak{S}_k$ coincides with the inclusion of the trivial subgroup $\{1_k\} \subseteq \mathfrak{S}_k$.  By Frobenius reciprocity, the left Kan extension $(\functor{h}_k)_! \mathbb{Q}$ is more-commonly written as the induced module $\mathrm{Ind}_{\{1_k\}}^{\mathfrak{S}_k} \mathbb{Q}$, which is the regular representation $\mathbb{Q} \mathfrak{S}_k$.  From the representation theory of finite groups, we recall the decomposition
$$
\mathbb{Q} \mathfrak{S}_k \cong \bigoplus_{|\lambda| = k} \mathcal{W}(\lambda)^{\oplus \left( \dim \mathcal{W}(\lambda) \right)},
$$
from which we obtain the result using additivity of $(\functor{i}_k)_!$ and the sizes of the action matrices for $\mathcal{W}(\lambda)$.
\end{proof}

The results of this paper rely on the following theorem of Sam-Snowden.

\begin{thm}[\cite{SS16} Corollary 4.2.5] \label{thm:ssinjective}
The induced module $\mathcal{M}(\lambda)$ is an injective object in the abelian category of finitely generated $\FI$-modules over $\mathbb{Q}$.  
\end{thm}

\begin{rem}[Nagpal's theorem on semi-induced shifts]
Even over $\mathbb{Z}$, the induced modules remain building-blocks for $\FI$-modules.  Nagpal's theorem says that every finitely generated $\FI$-module has some ``shift'' which is semi-induced, meaning that it has a filtration whose associated graded is an induced module.  For details on this fundamental result, see \cite{Nagpal15} or \cite{NSS17}.
\end{rem}

\subsection{Torsion and torsion-free $\FI$-modules}
An element $m \in M[x]$ of an $\FI$-module is called torsion if there exists some injection $f \colon [x] \to [y]$ so that $mf=0$.  An $\FI$-module is called \textbf{torsion} if all of its elements are torsion, and is called \textbf{torsion-free} if all of its torsion elements vanish.

\begin{prop} \label{prop:mtf}
The free modules $\F^k$ are torsion-free, and consequently the induced modules $\mathcal{M}(\lambda)$ are torsion-free as well.
\end{prop}
\begin{proof}
Every arrow of $\FI$ is monic, so post-composition by $f \colon [x] \to [y]$ is always an injection.  Consequently, if $m = \sum_{i=1}^r \alpha_i g_i \in \F^k[x]$ for some distinct injections $g_1, \ldots, g_r \colon [k] \to [x]$ and nonzero scalars $\alpha_1, \ldots, \alpha_r \in \mathbb{Q}$, then the expression 
$$
mf = \sum_{i=1}^r \alpha_i (f \circ g_i) \in \F^k[y]
$$
can have no cancellation since the injections $(f \circ g_i)$ remain distinct, and so if $mf=0$, then $r=0$, and so $m=0$.

By Proposition \ref{prop:summand}, the induced module $\mathcal{M}(\lambda)$ is a summand of the free module $\F^k$ for $k=|\lambda|$.  This proves the claim since any submodule of a torsion-free module is torsion-free.
\end{proof}

\begin{prop} \label{prop:tors}
If $T, F$ are $\FI$-modules with $T$ torsion and $F$ torsion-free, then $\Hom(T, F) = 0$.
\end{prop}
\begin{proof}
Let $\varphi \colon T \to F$, and suppose $t \in T[x]$ for some $x$.  Since $T$ is torsion, there exists some injection $f \colon [x] \to [y]$ so that $tf = 0$.  However, since $F$ is torsion-free, $\varphi_x(t)f = 0$ implies $\varphi_x(t)=0$, and so $\varphi$ maps every element to zero.
\end{proof}

\begin{prop} \label{prop:fgtors}
If $T$ is torsion and finitely generated, then $T[n]=0$ for all $n \gg 0$.
\end{prop}
\begin{proof}
Take $n$ large enough so that every generator is killed.
\end{proof}

\section{Explicit construction of induced $\FI$-modules} \label{sec:induced}
\noindent
Suppose $\mathcal{A}$ is an additive category with composition written $\cdot$, and that
$$
W \colon \mathfrak{S}_k \to \mathcal{A}
$$
is a functor.  In our application, $\mathcal{A}$ will be the category whose morphisms are matrices over $\mathbb{Q}$ and where composition is usual matrix multiplication.  Concretely, the functoriality assumption asserts that, if $\sigma, \tau \in \mathfrak{S}_k$,
$$
 W(\sigma) \cdot W(\tau) = W(\tau \circ \sigma).
$$
Our goal in this section is to describe the induced module 
$$
(\functor{i}_k)_!W \colon \FI \to \mathcal{A}.
$$
For any injection $f \in \FI(x, y)$, define $\nu(f) = f \circ \xi(f)^{-1}$, recalling that $\xi$ is defined so that $\nu(f) \in \OI(x, y)$.  The defining property is equivalent to
\begin{equation} \label{eq:nu}
f = \nu(f) \circ \xi(f).
\end{equation}
Build an $\OI(k,x) \times \OI(k,y)$ block matrix $V(f)$ with $(p,q)$-block entry given by
$$
V(f)_{p,q} = \begin{cases*}
W(\xi(f \circ p)) & if $\nu(f \circ p) = q$ and \\
0 & otherwise.
\end{cases*}
$$
Note that the condition $\nu(f \circ p) = q$ is equivalent to the condition $\mathrm{im}(f \circ p) = \mathrm{im}(q)$ since there is a unique monotone injection with specified image.  Also for this reason, if $\sigma \in \mathfrak{S}_x$, then $\nu(f \circ \sigma) = \nu(f)$, from which we deduce the useful property
\begin{equation} \label{eq:gobble}
\xi(f) \circ \sigma = \xi(f \circ \sigma).
\end{equation}

In the next two lemmas, let $f \colon [x] \to [y]$ and $g \colon [y] \to [z]$ be injections, and let $p \in \OI(k, x)$, $q \in \OI(k, y)$, and $r \in \OI(k,z)$ be monotone injections.
\begin{lem} \label{lem:pqr}
For all $p$ and $r$,
$$
\begin{array}{ccc}
\nu(g \circ f \circ p) = r & \Leftrightarrow & \begin{array}{c} \mbox{there exists a unique $q$ so that} \\ \nu(f \circ p) = q \mbox{ and } \nu(g \circ q) = r. \\ \end{array} \\
\end{array}
$$
\end{lem}
\begin{proof}
Suppose $\nu(g \circ f \circ p) = r$, set $q = \nu(f \circ p)$, and compute
\begin{align*}
\nu(g \circ q) &= \nu(g \circ q \circ \xi(f \circ p)) \\
&= \nu(g \circ \nu(f \circ p) \circ \xi(f \circ p)) \\
&= \nu(g \circ f \circ p) \\
&= r,
\end{align*}
using (\ref{eq:gobble}), the definition of $q$, and (\ref{eq:nu}).  Similarly, supposing $\nu(f \circ p) = q$ and $\nu(g \circ q) = r$, compute
\begin{align*}
\nu(g \circ f \circ p) &= \nu(g \circ \nu(f \circ p) \circ \xi(f \circ p)) \\
&= \nu(g \circ q \circ \xi(f \circ p)) \\
&= \nu(g \circ q) \\
&= r. \qedhere
\end{align*}
\end{proof}
\begin{lem} \label{lem:rep} $V(f) \cdot V(g) = V(g \circ f)$.
\end{lem}
\begin{proof}
We show that corresponding blocks are equal.
\begin{align*}
\left[ V(f) \cdot V(g) \right]_{p, r} &= \sum_{q \in \OI(k, y)} V(f)_{p,q} \cdot V(g)_{q,r} \\
&= W(\xi(f \circ p)) \cdot W(\xi(g \circ \nu(f \circ p))) \\
&= W(\xi(g \circ \nu(f \circ p)) \circ \xi(f \circ p))) \\
&= W(\xi(g \circ \nu(f \circ p) \circ \xi(f \circ p))) \\
&= W(\xi(g \circ f \circ p)) \\
&= V(g \circ f)_{p,r},
\end{align*}
where we have used the usual formula for matrix multiplication, Lemma \ref{lem:pqr}, the definition of $V$, functoriality of $W$, (\ref{eq:gobble}), (\ref{eq:nu}), and the definition of $V$.
\end{proof}
According to Lemma \ref{lem:rep}, the matrices $V(f)$ fit together to define a functor $\FI \to \mathcal{A}$.  We write $V$ for the resulting $\FI$-module, and show that it has the universal property characterizing the induced module $(\functor{i}_k)_!W$.
\begin{thm} \label{thm:lan}
If $M$ is any $\FI$-module, restriction to degree $k$ gives a natural isomorphism
$$
\Hom(V, M) \overset{\sim}{\longrightarrow} \Hom(W, (\functor{i}_k)^* M)
$$
compatible with maps $M \to M'$.
\end{thm}
\begin{proof}
Suppose $\psi \colon V \to M$ is a map of $\FI$-modules.  Such a map consists of a component $\psi(n)$ for every object $[n] \in \FI$.  Using this notation, the purported isomorphism takes $\psi$ to its component at $k$:
$$
\psi(k) \colon V[k] \to M[k],
$$
noting that $V[k] = W^{\oplus \OI(k,k)} = W$ since there is only one monotone injection $[k] \to [k]$.

In general, $V[n] = W^{\oplus \OI(k,n)}$, and so the component $\psi(n)$ takes the form of a column block matrix with rows indexed by $\OI(k,n)$ and block entries written $\psi(n)_{p,1} \colon W \to M[n]$.  Since $\psi$ is a map of $\FI$-modules, its components satisfy
$$
\psi(x) \cdot M(f) = V(f) \cdot \psi(y)
$$
for any injection $f \colon [x] \to [y]$.

Let us pause to examine the block matrix $V(o)$ for $o \in \OI(k,n)$.  Since $\OI(k,k)=\{1_k\}$, there is only one row.  The condition $\nu(o \circ 1_k) = q$ is equivalent to $o=q$, and so $V(o)$ has a single nonzero block entry, occurring in position $(1_k, o)$.  Moreover, $V(o)_{1_k,o} = W(1_k)$.  It follows that $V(o) \cdot \psi(n) = \psi(n)_{o,1}$.

Using this observation, and the compatibility equation $\psi(k) \cdot M(o) = V(o) \cdot \psi(n)$, we show that every entry of $\psi(n)$ may be recovered from the all-important $\psi(k)$:
$$
\psi(n)_{o,1} = V(o) \cdot \psi(n) = \psi(k) \cdot M(o).
$$
It remains to show that any choice of $\psi(k)$ compatible with the action of $\mathfrak{S}_k$ gives rise to a system of maps $\psi(n)$ that are compatible with the action of $\FI$.

Let $f \colon [x] \to [y]$, suppose $p \in \OI(k,x)$, and compute
\begin{align*}
\left[ V(f) \cdot \psi(y) \right]_{p,1}  &= \sum_{q \in \OI(k,y)} V(f)_{p,q} \cdot \psi(y)_{q,1}\\
&= V(f)_{p,\nu(f \circ p)} \cdot \psi(y)_{\nu(f \circ p),1}\\
&= W(\xi(f \circ p)) \cdot \psi(k) \cdot M(\nu(f \circ p)) \\
&= \psi(k) \cdot M(\xi(f \circ p)) \cdot M(\nu(f \circ p)) \\
&= \psi(k) \cdot M( \nu(f \circ p) \circ \xi(f \circ p) ) \\
&= \psi(k) \cdot M( f \circ p ) \\
&= \psi(k) \cdot M(p) \cdot M( f ) \\
&= \psi(x)_{p,1} \cdot M( f ) \\
&=  \left[\psi(x) \cdot M(f)\right]_{p,1} \qedhere
\end{align*}
\end{proof}
\noindent
From Theorem \ref{thm:lan} we obtain a well-known corollary; see \cite[(4)]{CEF15}.
\begin{cor} \label{cor:ind}
If $W$ is an $\mathfrak{S}_k$-representation and $n \geq k$, there is an isomorphism of $\mathfrak{S}_n$-representations
$$
(\functor{i}_n)^* (\functor{i}_k)_! W \cong \mathrm{Ind}_{\mathfrak{S}_k \times \mathfrak{S}_{n-k}}^{\mathfrak{S}_n} W \boxtimes \mathbf{1}_{\mathfrak{S}_{n-k}},
$$
where $\mathbf{1}_{\mathfrak{S}_{n-k}}$ denotes the trivial representation.
\end{cor}

\section{The infinite diagram $\lambda^+$ and its properties} \label{sec:infinite}
We make precise the previously-indicated idea of a diagram $\lambda^+$ consisting of $\lambda$ hanging below an infinitely-long top row.

A subset $S \subseteq \mathbb{N}_+^2$ is called a \textbf{horizontal strip} if it contains at most one element in each column.  For example, the top row $\{(1,i) : i \in \mathbb{N}_+\}$ is an (infinite) horizontal strip.  Despite its appearance, the subset $\{(1,5), (6,10)\}$ is also a horizontal strip, so this standard terminology is perhaps not self-evident.

If $\lambda, \lambda'$ are diagrams with $\lambda \subseteq \lambda'$, and if $(\lambda') \setminus \lambda$ is a horizontal strip, then we say that $\lambda'$ is a \textbf{horizontal-strip-extension} of $\lambda$, written $\lambda \subseteq_h \lambda'$.  This notation makes it easy to state a classical result about induced representations.
\begin{thm}[Pieri's rule]
If $\lambda$ is a diagram of size $k$, there is an isomorphism of $\mathfrak{S}_{k+n}$-representations.
$$
\mathrm{Ind}_{\mathfrak{S}_k \times \mathfrak{S}_n}^{\mathfrak{S}_{k+n}} \, \mathcal{W}(\lambda) \boxtimes \mathbf{1}_{\mathfrak{S}_{n}} \cong \bigoplus_{\substack{|\lambda'| = k+n \\ \lambda \subseteq_h \lambda'}} \mathcal{W}(\lambda')
$$
where $\mathbf{1}_{\mathfrak{S}_{n}}$ denotes the trivial representation of $\mathfrak{S}_n$.
\end{thm} 

From Corollary \ref{cor:ind} we recognize $\mathrm{Ind}_{\mathfrak{S}_k \times \mathfrak{S}_n}^{\mathfrak{S}_{k+n}} \, \mathcal{W}(\lambda) \boxtimes \mathbf{1}_{\mathfrak{S}_{n}} $ as the degree $k+n$ part of the induced module $\mathcal{M}(\lambda)$, and so
\begin{equation} \label{eq:mlambda}
\mathcal{M}(\lambda)[k+n] \cong \bigoplus_{\substack{|\lambda'| = k+n \\ \lambda \subseteq_h \lambda'}} \mathcal{W}(\lambda').
\end{equation}

We write $\lambda^+$ for the maximal horizontal-strip-extension of $\lambda$, which is the union of all $\lambda'$ with $\lambda \subseteq_h \lambda'$, and is explicitly given by the formula
$$
\lambda^+ = \left\{ (i, j) \in \mathbb{N}_+^2 \mbox{ so that $(i-1,\,j) \in \lambda$ or $i = 1$} \right\}.
$$
In other words, $\lambda^+$ is $\lambda$ with an extra box in each column---even in the empty columns.  Consequently, the difference $(\lambda^+) \setminus \lambda$ is a horizontal strip, and so we always have $\lambda \subseteq_h (\lambda^+)$.  

As a finite approximation to the infinite diagram $\lambda^+$, define
$$
\lambda^{+n} = \lambda^+ \cap \left\{ (i, j) \mbox{ with $j \leq n$} \right\}.
$$
Once $n \geq \lambda_1$, the diagram $\lambda^{+n}$ is $\lambda$ with a new top row of length $n$.  Consequently, for large $n$, $|\lambda^{+n}| = |\lambda| + n$ and $\lambda^{+n} \subseteq_h \lambda^+$.  
\begin{defn} \label{defn:mu}
Let $\lambda$ be a finite diagram, and set $k=|\lambda|$.  The multiplicity of the infinite diagram $\lambda^+$ in an $\FI$-module $M$ is defined as the limit
$$
\mu(\lambda^+, M) = \lim_{n \to \infty} \dim_{\mathbb{Q}} \Hom_{\mathfrak{S}_{k + n}}\left(\mathcal{W}(\lambda^{+n}), M[k+n] \right).
$$
\end{defn}
\begin{lem} \label{lem:diagrams}
The following are equivalent for any pair of finite diagrams $\lambda \subseteq \lambda'$:
\begin{enumerate}
\item $\lambda \subseteq_h \lambda'$
\item $\lambda' \subseteq_h (\lambda^+)$
\item for all $n \geq (\lambda')_1$, $\lambda' \subseteq_h (\lambda^{+n})$.
\end{enumerate}
\end{lem}
\begin{proof}
\textbf{(1) $\implies$ (2): } The diagrams $\lambda$ and $\lambda'$ are both subsets of $\lambda^+$.  Since $\lambda \subseteq \lambda'$, $(\lambda^+) \setminus \lambda' \subseteq (\lambda^+) \setminus \lambda$.
The larger set is a horizontal strip, so the subset is as well.

\textbf{(2) $\implies$ (3): } As $\lambda'$ is contained in the first $(\lambda')_1$ columns, it is also contained in the first $n$, and so $\lambda' \subseteq \lambda^{+n}$.  Then, since $(\lambda^+) \setminus \lambda'$ is a horizontal strip, the subset $(\lambda^{+n}) \setminus \lambda'$ is also a horizontal strip.

\textbf{(3) $\implies$ (1): } Take $n = (\lambda')_1$.  Since $\lambda' \subseteq (\lambda^{+n})$, $\lambda' \setminus \lambda \subseteq (\lambda^{+n}) \setminus \lambda$.  This last set is a horizontal strip of size $n$, so $\lambda' \setminus \lambda$ is a subset of a horizontal strip.
\end{proof}

\begin{prop} \label{prop:mu} 
Let $\lambda, \lambda'$ be diagrams of sizes $k, k' \in \mathbb{N}$.  We have
$$
\mu\left(\lambda^+, \mathcal{M}(\lambda') \right) = \dim_{\mathbb{Q}} \Hom(\mathcal{M}(\lambda'), \mathcal{M}(\lambda)).
$$
\end{prop}
\begin{proof}
Use Pieri, Lemma \ref{lem:diagrams}, Pieri again, (\ref{eq:mlambda}), and the universal property of $\mathcal{M}(\lambda')$:
\begin{align*}
\mu(\lambda^+, M(\lambda')) &= \lim_{n \to \infty} \dim_{\mathbb{Q}} \Hom_{\mathfrak{S}_{k + n}}(\mathcal{W}(\lambda^{+n}), \mathcal{M}(\lambda')[k+n]) \\
&= \lim_{n \to \infty} \; \begin{cases*} 
1 & if $\lambda' \subseteq_h \lambda^{+n}$ \\
0 & otherwise.
\end{cases*} \\
&= \begin{cases*} 
1 & if $\lambda' \subseteq_h \lambda^{+}$ \\
0 & otherwise.
\end{cases*} \\
&=
\begin{cases*} 
1 & if $\lambda \subseteq_h \lambda'$ \\
0 & otherwise.
\end{cases*} \\
&= \dim_{\mathbb{Q}} \Hom_{\mathfrak{S}_{k'}}(\mathcal{W}(\lambda'), \; \mathrm{Ind}_{\mathfrak{S}_k \times \mathfrak{S}_{k-k'}}^{\mathfrak{S}_{k'}} \, \mathcal{W}(\lambda) \boxtimes \mathbf{1}_{\mathfrak{S}_{k'-k}} ) \\
&= \dim_{\mathbb{Q}} \Hom_{\mathfrak{S}_{k'}}(\mathcal{W}(\lambda'), \mathcal{M}(\lambda)[k'] ) \\
&= \dim_{\mathbb{Q}} \Hom(\mathcal{M}(\lambda'), \mathcal{M}(\lambda) ). \qedhere
\end{align*}
\end{proof}

\begin{lem} \label{lem:mu}
If $M$ is a finitely presented $\FI$-module over the rational numbers, then the multiplicity of $\lambda^+$ is given by
$$
\mu(\lambda^+, M) = \mathrm{dim}_{\mathbb{Q}} \Hom(M, \mathcal{M}(\lambda)).
$$
\end{lem}
\begin{proof}
Note that the function $\mu(\lambda^+, -)$ is additive in short exact sequences.  This lets us use a result of Sam-Snowden that gives a basis for the K-theory of finitely generated $\FI$-modules over $\mathbb{Q}$.

Specifically, we rely on \cite[Proposition 4.9.2]{SS16}, which implies that any function additive in short exact sequences is determined by its values on the induced modules $\mathcal{M}(\lambda')$, and another collection of finitely generated $\FI$-modules $\mathcal{I}(\lambda')$ that are torsion.

By Proposition \ref{prop:mtf}, the induced module $\mathcal{M}(\lambda)$ is torsion-free, and so by Proposition \ref{prop:tors}, $\Hom(\mathcal{I}(\lambda'), \mathcal{M}(\lambda)) = 0$ since $\mathcal{I}(\lambda')$ is torsion.  Consequently,
$$
\mu(\lambda^+, \mathcal{I}(\lambda')) = 0 = \dim_{\mathbb{Q}} \Hom(\mathcal{I}(\lambda'), \mathcal{M}(\lambda))
$$
since all eventual multiplicities vanish in the finitely-generated torsion $\FI$-module $\mathcal{I}(\lambda')$ using Proposition \ref{prop:fgtors}.  On the other hand, 
$$
\mu(\lambda^+, \mathcal{M}(\lambda')) = \dim_{\mathbb{Q}} \Hom(\mathcal{M}(\lambda'), \mathcal{M}(\lambda))
$$
by Proposition \ref{prop:mu}.  According to a result of Sam-Snowden, which we have given as Theorem \ref{thm:ssinjective}, the modules $\mathcal{M}(\lambda)$ are injective in the category of finitely generated $\FI$-modules, and so the function $\dim_{\mathbb{Q}} \Hom(-,\mathcal{M}(\lambda))$ is additive in short exact sequences.  It then follows from \cite[Proposition 4.9.2]{SS16} that these two functions coincide for all finitely presented $\FI$-modules $M$, as required.
\end{proof}

\section{Proof of Theorem \ref{thm:main}} \label{sec:proof}
\begin{proof}
By Theorem \ref{thm:lan}, the explicitly-defined module $V$ provides a model for the induced module $(\functor{i}_k)_! W$ where $W$ is any $\mathfrak{S}_k$-representation.  Using the formula for the action of $\mathfrak{S}_k$ on the Specht module $\mathcal{W}(\lambda)$ given in (\ref{eq:specht}), we obtain a formula for the action of an injection $f \colon [x] \to [y]$ on the induced module $\mathcal{M}(\lambda) = (\functor{i}_k)_! \mathcal{W}(\lambda)$:
$$
\left( \mathcal{M}(\lambda) \right)(f) = \mathbb{A}_{\lambda}(1_x)^{-1} \cdot \mathbb{A}_{\lambda}(f).
$$
Applying the functor $\Hom(-, \mathcal{M}(\lambda))$ to the presentation
$$
\bigoplus_{j = 1}^r \F^{y_j} \overset{Z}{\longrightarrow} \bigoplus_{i=1}^g \F^{x_i} \longrightarrow M \longrightarrow 0
$$
gives the exact sequence
$$
0 \longrightarrow \Hom(M, \mathcal{M}(\lambda)) \longrightarrow \bigoplus_{i=1}^g \Hom(\F^{x_i}, \mathcal{M}(\lambda)) \overset{Z^*}{\longrightarrow} \bigoplus_{j=1}^r \Hom(\F^{y_j}, \mathcal{M}(\lambda)),
$$
which simplifies by Yoneda's lemma:
$$
0 \longrightarrow \Hom(M, \mathcal{M}(\lambda)) \longrightarrow \bigoplus_{i=1}^g \left( \mathcal{M}(\lambda) \right)[x_i] \xrightarrow{\left( \mathcal{M}(\lambda) \right)(Z)} \bigoplus_{j=1}^r  \left( \mathcal{M}(\lambda) \right)[y_j].
$$
As a result,
\begin{align*}
\mathrm{dim}_{\mathbb{Q}} \Hom(M, \mathcal{M}(\lambda)) &= \mathrm{corank} \left( \mathcal{M}(\lambda) \right)(Z) \\
&= \mathrm{corank} \left[ \left( \bigoplus_{i=1}^g \mathbb{A}_{\lambda}(1_{x_i}) \right)^{-1} \cdot \mathbb{A}_{\lambda}Z \right] \\
&= \mathrm{corank} \left( \mathbb{A}_{\lambda}Z \right).
\end{align*}
By Lemma \ref{lem:mu}, $\mu(\lambda^+, M) = \mathrm{dim}_{\mathbb{Q}} \Hom(M,  \mathcal{M}(\lambda))$, and we are done.
\end{proof}

\newpage
\section{Appendix: Sage code} \label{sec:code}
\begin{verbatim}
def circ(g, f):
    return [g[v - 1] for v in f]

def zeta(tableau):
    return Permutation([entry for row in tableau for entry in row])

def xi(injection):
    monotonic = sorted(injection)
    return [monotonic.index(i) + 1 for i in injection]

def chi(r, c):
    boxes = zip(r, c)
    if len(boxes) != len(set(boxes)):
        return 0
    lex = sorted(boxes)
    perm = [lex.index(b) + 1 for b in boxes]
    return Permutation(perm).signature()

def row_word(tableau):
    k = sum([len(row) for row in tableau])
    return [i + 1 for l in range(k)
                  for i, row_i in enumerate(tableau)
                  for entry in row_i if entry == l + 1]

def col_word(tableau):
    k = sum([len(row) for row in tableau])
    return [j + 1 for l in range(k) for row in tableau
                  for j, entry in enumerate(row) if entry == l + 1]

def AA_entry(f, (p, t), (q, u)):
    a, b = circ(row_word(u), xi(circ(f, p))), col_word(t)
    return chi(a, b) if sorted(circ(f, p)) == q else 0

def OI(k, n):
    return [sorted(s) for s in Subsets(range(1, n + 1), k)]

def AA(partition, x, f, y):
    k = sum(partition)
    tableaux = list(StandardTableaux(partition))
    rows = [(p, t) for p in OI(k, x) for t in tableaux]
    cols = [(q, u) for q in OI(k, y) for u in tableaux]
    entries = [AA_entry(f, pt, qu) for pt in rows for qu in cols]
    return matrix(QQ, len(rows), len(cols), entries)
\end{verbatim}

\bibliographystyle{amsalpha}
\bibliography{references}
\end{document}